\documentclass[12pt]{amsart}
\usepackage{a4wide}
\usepackage{amssymb}
\usepackage{pdfsync}
\usepackage{wrapfig,epsfig}
\usepackage{color}

\newtheorem{lemma}{Lemma}
\newtheorem{proposition}{Proposition}
\newtheorem{theorem}{Theorem}

\makeatletter
\@addtoreset{equation}{section}

\newcommand{\bes}{\begin{displaymath}}
\newcommand{\ees}{\end{displaymath}}
\newcommand{\be}{\begin{equation}}
\newcommand{\ee}{\end{equation}}
\newcommand{\ba}{\begin{eqnarray}}
\newcommand{\ea}{\end{eqnarray}}
\newcommand{\bas}{\begin{eqnarray*}}
\newcommand{\eas}{\end{eqnarray*}}
\newcommand{\@Bbb}[1]{\ensuremath{\Bbb #1}}

\newcommand{\B}{{\@Bbb B}}
\newcommand{\C}{{\@Bbb C}}
\newcommand{\E}{{\@Bbb E}}
\newcommand{\F}{{\@Bbb F}}
\newcommand{\G}{{\@Bbb G}}
\renewcommand{\P}{{\@Bbb P}}

\newcommand{\Q}{{\@Bbb Q}}
\newcommand{\bQ}{{\@Bbb Q}}
\newcommand{\N}{{\@Bbb N}}
\newcommand{\R}{{\@Bbb R}}
\newcommand{\T}{{\@Bbb T}}
\newcommand{\bbR}{{\@Bbb R}}
\newcommand{\W}{{\@Bbb W}}
\newcommand{\Z}{{\@Bbb Z}}
\newcommand{\bbZ}{{\@Bbb Z}}

\newcommand{\@s}[1]{\ensuremath{\mathcal #1}}
\newcommand{\cA}{\@s A}
\newcommand{\cB}{\@s B}
\newcommand{\cC}{\@s C}
\newcommand{\cD}{\@s D}
\newcommand{\cE}{\@s E}
\newcommand{\cF}{\@s F}
\newcommand{\cG}{\@s G}
\newcommand{\cH}{\@s H}
\newcommand{\cI}{\@s I}
\newcommand{\cJ}{\@s J}

\newcommand{\cK}{\@s K}
\newcommand{\cL}{\@s L}
\newcommand{\cN}{\@s N}
\newcommand{\cM}{\@s M}
\newcommand{\cO}{\@s O}
\newcommand{\cP}{\@s P}
\newcommand{\cQ}{\@s Q}
\newcommand{\cR}{\@s R}
\newcommand{\cS}{\@s S}
\newcommand{\cT}{\@s T}
\newcommand{\cU}{\@s U}
\newcommand{\cV}{\@s V}
\newcommand{\cW}{\@s W}
\newcommand{\cX}{\@s X}
\newcommand{\cY}{\@s Y}
\newcommand{\cZ}{\@s Z}

\def\d{{\rm d}}

\newcommand{\@bm}[1]{\ensuremath{\mathbf #1}}
\newcommand{\bma}{\@bm a}\newcommand{\bmA}{\@bm A}
\newcommand{\bmb}{\@bm b}\newcommand{\bmB}{\@bm B}
\newcommand{\bmc}{\@bm c}\newcommand{\bmC}{\@bm C}
\newcommand{\bmd}{\@bm d}\newcommand{\bmD}{\@bm D}
\newcommand{\bme}{\@bm e}
\newcommand{\bmf}{\@bm f}\newcommand{\bmF}{\@bm F}
\newcommand{\bmg}{\@bm g}\newcommand{\bmG}{\@bm G}
\newcommand{\bmh}{\@bm h}\newcommand{\bmH}{\@bm H}
\newcommand{\bmi}{\@bm i}\newcommand{\bmI}{\@bm I}
\newcommand{\bmj}{\@bm j}
\newcommand{\bmk}{\@bm k}\newcommand{\bmK}{\@bm K}
\newcommand{\bml}{\@bm l}
\newcommand{\bmm}{\@bm m}\newcommand{\bmM}{\@bm M}
\newcommand{\bmn}{\@bm n}
\newcommand{\bmo}{\@bm o}
\newcommand{\bmp}{\@bm p}
\newcommand{\bmq}{\@bm q}\newcommand{\bmQ}{\@bm Q}
\newcommand{\bmr}{\@bm r}
\newcommand{\bms}{\@bm s}\newcommand{\bmS}{\@bm S}
\newcommand{\bmt}{\@bm t}
\newcommand{\bmu}{\@bm u}\newcommand{\bmU}{\@bm U}
\newcommand{\bmw}{\@bm w}\newcommand{\bmW}{\@bm W}
\newcommand{\bmv}{\@bm v}\newcommand{\bmV}{\@bm V}
\newcommand{\bmx}{\@bm x}\newcommand{\bmX}{\@bm X}\newcommand{\bx}{\@bm x}
\newcommand{\bmy}{\@bm y}\newcommand{\bmY}{\@bm Y}\newcommand{\by}{\@bm y}
\newcommand{\bmz}{\@bm z}\newcommand{\bmZ}{\@bm Z}

\newcommand{\bmzero}{\@bm 0}

\newcommand{\@g}[1]{\ensuremath{\mathfrak #1}}
\newcommand{\gA}{\@g A}
\newcommand{\gD}{\@g D}
\newcommand{\gJ}{\@g J}
\newcommand{\gF}{\@g F}
\newcommand{\gM}{\@g M}
\newcommand{\gR}{\@g R}

\newcommand{\Lip}{\mathop{\mbox{Lip}}}

\newcommand{\commentout}[1]{{}}

\makeatother

\begin{document}

\title[An invariance principle...]{An invariance principle for the law of the iterated logarithm for some Markov chains}
\author{W. Bo\l t}
\address{W.B. Institute of Mathematics, University of Gda\'nsk, Wita Stwosza
57, 80-952 Gda\'nsk, Poland}
\email{ja@hope.art.pl}
 \author{A. A. Majewski }
\address{A.A.M. Institute of Mathematics, University of Gda\'nsk, Wita Stwosza
57, 80-952 Gda\'nsk, Poland}
 \email{aamajewski@gmail.com}

\author{T. Szarek}
\address{T.S. Institute of Mathematics, University of Gda\'nsk, Wita Stwosza
  57, 80-952 Gda\'nsk, Poland}
\email{szarek@intertele.pl}

\begin{abstract} 
Strassen's invariance principle for additive functionals of Markov chains with spectral gap in the Wasserstein metric is proved.
\end{abstract}

\subjclass[2000]{60J25, 60H15 (primary), 76N10 (secondary)}
\keywords{Ergodicity of Markov families, invariant measures, law of iterated logarithm}

\thanks{Tomasz Szarek has been supported by Polish Ministry of Science and
Higher Education  Grants N N201 419139. }
\date\today
\maketitle

\section{Introduction}
Suppose that $(E, \rho)$ is a Polish space. By $\mathcal B(E)$ we denote the family of all Borel sets in $E$. By $\mathcal M_1$ we denote the space of all probability Borel measures on $E$. Let $\pi: E\times\mathcal B(E)\to [0, 1]$ be a transition probability on $E$. The Markov operator $P$ is defined by $Pf(x)=\int_E f(y)\pi (x, \d y)$ for every bounded Borel measurable function $f$ on $E$. The same formula defines $Pf$ for any Borel measurable functions $f\ge 0$  which need not be finite. Denote by $B_b(E)$ the set of all bounded Borel measurable functions equipped with the supremum norm and let $C_b(E)$ be its subset consisting of all bounded continuous functions.

Suppose that $(X_n)_{n\ge 0}$ is an $E$--valued Markov chain, given over some probability space $(\Omega, \mathcal F, \mathbb P)$, whose transition is $\pi$ and its initial distribution is equal to $\mu_0$. Denote by $\mathbb E$ the expectation corresponding to $\mathbb P$. We shall denote by $\mu P$, the associated transfer operator describing the evolution of the law of $X_n$. To be precise, $\mu P$ is defined by the formula $\int_E f(x)\mu P(\d x)=\int_E Pf(x)\mu (\d x)$ for any $f\in B_b(E)$ and $\mu\in\mathcal M_1$. To simplify the notation we shall write $\langle f, \mu\rangle$ instead of $\int_E f(y)\mu (\d y)$.

Given a Lipschitz function $\psi: E\to\mathbb R$ we define
$$
S_n(\psi):= \psi(X_0)+\psi(X_1)+\ldots+\psi (X_n)\quad\text{for $n\ge 0$}.
$$
Our aim is to find conditions under which $S_n(\psi) $ satisfies the law of the iterated logarithm {\bf (LIL)}. This natural question is raised when central limit theorems {\bf (CLT)} are verified. Since 1986 when Kipnis and Varadhan \cite{KV} proved the central limit theorem for additive functionals of stationary reversible ergodic Markov chains, it has been a huge amount of reviving attempts to do this in various settings and under different conditions (see \cite{M, MW}). A common factor of the mentioned results was that they  were established with respect to the stationary probability law of the chain. In \cite{DL} Derriennic and Lin answered the question about the validity of the {\bf CLT} with
respect to the law of the Markov chain starting at some point $x$. Namely, they proved that the {\bf CLT} holds for almost every
$x$ with respect to the invariant initial distribution (see also \cite{DL1} and the references therein). On the other hand, Guivarc'h and Hardy \cite{GH} proved the {\bf CLT} for a class of Markov chains associated with the transfer operator having spectral gap. Recently Komorowski and Walczuk studied Markov processes with the transfer operator having spectral gap in the Wasserstein metric and proved the {\bf CLT} in the non-stationary case (see \cite{KW}).
Other interesting results under similar assumptions were obtained by  S. Kuksin and A. Shirikyan (see \cite{K, Sh1}).

The {\bf LIL} we study in this note was also considered in many papers. There are several results governing, for instance, the Harris recurrent chains \cite{Ch, Ch-1, MT}. Similarly the {\bf CLT} results they are formulated mostly for stationary ergodic chains (see for instance \cite{A, CC, LW, R, WZ}). In the case when one is able to find the solution to the Poisson equation $h=f+Ph$, the problem may be reduced to the martingale case \cite{GL} (see also \cite{MT}). But the {\bf LIL} for martingales was carefully examined in many papers (see \cite{HH, HS, St-1, St-2}) and a lot of satisfactory results were obtained.

Our note is aimed
at proving the {\bf LIL} for Markov chains that satisfy the spectral gap property in the Wasserstein metric. It is worth mentioning here that many Markov chains satisfy this property, e.g. Markov chains associated with iterated function systems or stochastic differential equations disturbed with Poisson noise (see \cite{L}). 

Our result is based upon the {\bf LIL} for martingales due to Heyde and Scott (see Theorem 1 in \cite{HS}).

\section{Assumptions and auxiliary results}

For every measure $\nu\in\mathcal M_1$ the law of the Markov chain $(X_n)_{n\ge 0}$ with transition probability $\pi$ and initial distribution $\nu$, is the probability measure $\mathbb P_{\nu}$ on $(E^{\mathbb N}, \mathcal B(E)^{\otimes\mathbb N})$ such that:
$$
\mathbb P_{\nu}[X_{n+1}\in A|X_n=x]=\pi (x, A)\quad\text{and}\quad\mathbb P_{\nu}[X_0\in A]=\nu(A),
$$
where $x\in E$, $A \in \mathcal B(E)$. The expectation with respect to $\mathbb P_{\nu}$ is denoted by $\mathbb E_{\nu}$. For $\nu=\delta_x$, the Dirac measure at $x\in E$, we write just $\mathbb P_x$ and $\mathbb E_x$.

We will make the following assumption:
\begin{itemize}
\item[{\bf (H0)}] the Markov operator satisfies the Feller property, i.e. $P (C_b(E))\subset C_b(E)$.
\end{itemize}

We shall denote by $\mathcal M_{1, 1}$ the space of all probability measures possessing finite first moment, i.e. $\nu\in\mathcal M_{1, 1}$ iff $\nu\in\mathcal M_1$ and $\int_E \rho(x_0, x)\nu(\d x)<\infty$ for some (thus all) $x_0\in E$. For abbreviation we shall write $\rho_{x_0}(x)=\rho(x_0, x)$. We assume that:

\begin{itemize}

\item[{\bf (H1)}] for any $\nu\in\mathcal M_{1, 1}$ we have $P\nu \in \mathcal M_{1, 1}$.
\end{itemize}
It may be proved that the space $\mathcal M_{1, 1}$ is a complete metric space when equipped with the Wasserstein metric
$$
d(\nu_1, \nu_2)=\sup\{|\langle f, \nu_1\rangle- \langle f, \nu_2\rangle|: f: E\to \mathbb R,\,\,\, \Lip f\le 1\}
$$
for $\nu_1, \nu_2\in\mathcal M_{1, 1}$ and the convergence in the Wasserstein metric is equivalent to the weak convergence, see e.g. \cite{V}. (Here $\Lip f$ denotes the Lipschitz constant of $f$.) The main assumption made in our note says that the Markov operator $P$ is contractive with respect to the Wasserstein metric, i.e.
\begin{itemize}
\item[{\bf (H2)}] there exist $\gamma\in (0, 1)$ and $c>0$ such that 
\begin{equation}
\label{eqh4}
d(\mu P^n, \nu P^n) \leq c  \gamma^ n d(\mu,\nu), \ \ \ \ \text{for}\,\, n\ge 1, \ \mu, \nu \in \mathcal{M}_{1, 1}.
\end{equation}
\end{itemize}

Let $\mu\in\mathcal M_{1, 1}$. From now on we shall assume that the initial distribution of $(X_n)_{n\ge 0}$ is $\mu$. Moreover,
\begin{itemize}
\item[{\bf (H3)}] there exists $x_0\in E$ and $\delta > 0$ such that 
\begin{equation}\label{H3}
\sup_{n\geq 0} \mathbb E_{\mu} \rho_{x_0}^{2+\delta}(X_n) < \infty.
\end{equation}
\end{itemize}

It is easy to prove that under the assumptions (H0)--(H3) there exists a unique invariant (ergodic) measure $\mu_*\in\mathcal M_{1}$. In particular, $\mu_*\in\mathcal M_{1, 1}$. The proof was given for Markov processes with continuous time in \cite{KW} but it still remains valid in discrete case.

Let $n_0\ge  2$ be such that 
$$
\gamma_0=c^2 \gamma^{n_0}<1.
$$

We start this part of the paper with a rather technical lemma.

\begin{lemma}

Let $g_{n, k}: E^{2(k+n)}\to\mathbb R$ for arbitrary $k, n\ge 1$, be Lipschitz continuous in each variable with the same Lipschitz constant $L$. Then there exists constant $\tilde L$ dependent only on  $L$ and such that the function
\begin{equation}\label{e1l}
\begin{aligned}
H_{n, k}(x)&=\int_{E}\pi_{{1}}(x, \d y_{1})\int_{E}\pi_{{2}}(y_{1}, \d y_{2})\cdots\int_{E}\pi_{2(k+n)-1}(y_{2(k+n)-2}, \d y_{2(k+n)-1})\\
&\times\int_{E}\pi_{2(k+n)}(y_{2(k+n)-1}, \d y_{2(k+n)})g_{n, k}(y_{1},\ldots, y_{2(k+n)}),
\end{aligned}
\end{equation}
where $\pi_l (y_{l-1}, \d y_l)=\delta_{y_{l-1}}P^{k_l}(\d y_l)$, $k_l\ge 1$ and additionally  $k_l\ge n_0-1$ for all even $l$, 
is Lipschitzean with the Lipschitz constant $\tilde L$.
\end{lemma}

\begin{proof} Define the functions $g_j: E^j\to\mathbb R$ by the formula
$$
\begin{aligned}
&g_j(y_0, y_1, \ldots, y_{j-1})=\int_{E}\pi_{{j}}(y_{j-1}, \d y_{j})\int_{E}\pi_{{j+1}}(y_{j}, \d y_{j+1})\times\cdots\\
&\times\int_{E}\pi_{2(k+n)}(y_{2(k+n)-1}, \d y_{2(k+n)})g_{n, k}(y_{1},\ldots, y_{2(k+n)})\,\,\,\text{for $j=1, \ldots, 2(k+n).$}
\end{aligned}
$$
Let $\mathcal L_{j, l}$  for  $j=1, \ldots, 2(k+n) $ and $l=0, \ldots, j-1$ denote the Lipschitz constant of $g_j$ with respect to $y_l$. Then the Lipschitz constant of $H_{n, k}$ is equal to $\mathcal L_{1, 0}$. It is obvious that $\mathcal L_{j, l}\le L$ for $0\le l<j-1$, $j>1$. To evaluate $\mathcal L_{j, j-1}$ fix $y_0, y_1,\ldots, y_{j-2}$ and $\tilde{y}_{j-1}, \hat {y}_{j-1}$. Then we have
$$
\begin{aligned}
&g_j(y_0, y_1, \ldots, y_{j-2}, \hat {y}_{j-1})-g_j(y_0, y_1, \ldots, y_{j-2}, \tilde{y}_{j-1})\\
&=\int_{E}\pi_{{j}}(\hat{y}_{j-1}, \d y_{j})g_{j+1}(y_0, y_1, \ldots, \hat {y}_{j-1}, y_{j})
-\int_{E}\pi_{{j}}(\tilde{y}_{j-1}, \d y_{j})g_{j+1}(y_0, y_1, \ldots, \tilde {y}_{j-1}, y_{j})\\
&= \int_{E}\pi_{{j}}(\hat{y}_{j-1}, \d y_{j}) (g_{j+1}(y_0, y_1, \ldots, \hat {y}_{j-1}, y_{j})
-g_{j+1}(y_0, y_1, \ldots, \tilde {y}_{j-1}, y_{j}))\\
&+\int_{E}\pi_{{j}}(\hat{y}_{j-1}, \d y_{j})g_{j+1}(y_0, y_1, \ldots, \tilde {y}_{j-1}, y_{j})
-\int_{E}\pi_{{j}}(\tilde{y}_{j-1}, \d y_{j})g_{j+1}(y_0, y_1, \ldots, \tilde {y}_{j-1}, y_{j})
\end{aligned}
$$
and consequently
$$
\begin{aligned}
&|g_j(y_0, y_1, \ldots, \hat {y}_{j-1})-g_j(y_0, y_1, \ldots, \tilde{y}_{j-1})|\le  \mathcal L_{j+1, j-1} \rho( \hat {y}_{j-1},  \tilde {y}_{j-1})\int_{E}\pi_{{j}}(\hat{y}_{j-1}, \d y_{j})\\
&+|\langle P^{k_{j}}\tilde{g}_{j+1}, \delta_{\hat{y}_j}\rangle - \langle P^{k_{j}}\tilde{g}_{j+1}, \delta_{\tilde{y}_j}\rangle |
\le L \rho( \hat {y}_{j-1},  \tilde {y}_{j-1})+c_j\mathcal L_{j+1, j} \rho( \hat {y}_{j-1},  \tilde {y}_{j-1}),
\end{aligned}
$$
where $c_j=c \gamma$ if $j$ odd, $c_j= c \gamma^{n_0-1}$ if $j$ even and $\tilde{g}_{j+1}(\cdot)=g_{j+1}(y_0, y_1, \ldots, y_{j-2}, \tilde{y}_{j-1}, \cdot)$. Hence we have
$$
\mathcal L_{j, j-1}\le  L +c_j\mathcal L_{j+1, j}\qquad\text{for $j=1, \ldots, 2(k+n) - 1$.}
$$
Since $\mathcal L_{2(k+n) , 2(k+n) -1}\le L$, an easy computation shows that
$$
\mathcal L_{1, 0}\le \frac{L(c\gamma+1)}{1-\gamma_0}.
$$
This completes the proof. 
\end{proof}

\section{The law of the iterated logarithm.}
\subsection{A martingale result}

We start with recalling a classical result due to C.C. Heyde and D.J. Scott \cite{HS}.
Let $\{S_n, \mathcal F_n: n\ge 0\}$ be a martingale on the probability space $(\Omega, \mathcal F, \mathbb P)$ where $\mathcal F_0=\{\Omega, \emptyset\}$ and $\mathcal F_n$ is the $\sigma$--field generated by $S_1, S_2,\ldots, S_n$ for $n>0$.
Let $S_0=Z_0=0$ $\mathbb P$-a.s. and $S_n=\sum_{k=1}^n Z_k$ for $n\ge 1$. Further, let $s_n^2=\mathbb ES_n^2<\infty.$

We consider the metric space $(C, \tilde \rho)$ of all real-valued continuous functions on $[0, 1]$ with
$$
\tilde\rho (x, y)=\sup_{0\le t\le 1}|x(t)-y(t)|\qquad\text{for $x, y\in C$}.
$$
Let $K$ be the set of absolutely continuous functions $x\in C$ such that $x(0)=0$ and $\int_0^1(x'(t))^2 \rm{ d} t\le 1$.

Define the real function $g$ on $[0, \infty)$ by $g(s)=\sup\{n: s_n^2\le s\}$.
We define a sequence of real random functions $\eta_n$ on $[0, 1]$, for $n> g(e)$,  by
$$
\eta_n(t)=\frac{S_k+(s_n^2t-s_k^2)(s_{k+1}^2-s_k^2)^{-1} Z_{k+1}}{\sqrt{2s_n^2\log\log s_n^2}}
$$ 
if $s_k^2\le s_n^2 t\le s_{k+1}^2$, $k=1, \ldots, n-1$ and 
$$
\eta_n(t)=0\qquad\text{for $n\le g (e)$.}
$$

\begin{proposition}
\label{pr1} {\bf (Theorem 1 in \cite{HS})}
If $s_n^2\to\infty$ and
\begin{equation}\label{e1}
\sum_{n=1}^{\infty} s_n^{-4}\mathbb E[Z_n^4{\bf 1}_{\{|Z_n|<\gamma s_n\}}]<\infty\qquad\text{for some $\gamma>0$},
\end{equation}
\begin{equation}\label{e2}
\sum_{n=1}^{\infty} s_n^{-1}\mathbb E[|Z_n|{\bf 1}_{\{|Z_n|\ge\epsilon s_n\}}]<\infty\qquad\text{for all $\epsilon>0$},
\end{equation}
\begin{equation}\label{e3}
s_n^{-2}\sum_{k=1}^n Z_k^2\to 1\quad \text{$\mathbb P$-a.s. as $n\to\infty$}
\end{equation}
hold, then $\{\eta_n\}_{n\ge 1}$ is relatively compact in $C$ and the set of its limit points coincides with $K$.
\end{proposition}

\subsection{Application to Markov chains}

Let $\psi :E\to \mathbb R$ be a Lipschitz function such that $\langle\psi, \mu_*\rangle=0$, otherwise we could consider $ \tilde{\psi}=\psi-\langle\psi, \mu_*\rangle $. Let $L>0$ denote its Lipschitz constant. Let $(X_n)_{n\ge 0}$ be a Markov chain with the initial distribution $\mu$ satisfying conditions (H0)--(H3).

We have $\sum_{i=0}^{\infty} |P^i\psi(x)| = \sum_{i=0}^{\infty} |\langle \psi, \delta_xP^i\rangle-\langle\psi, \mu_*P^i\rangle|\le
c d(\delta_x,\mu_*) \sum_{i=0}^{\infty}\gamma^i<\infty$, by (H2). Thus we may define the function
$$
\chi(x):=\sum_{i=0}^{\infty}P^i\psi(x)\qquad\text{for $x\in E$.}
$$
We easily check that $\chi$ is a Lipschitz function.

 It is well known that 
$$
S_n=\chi (X_n)-\chi (X_0)+\sum_{i=0}^n \psi(X_i)\qquad\text{for $n\ge 0$}
$$
is a martingale on the space $(E^{\mathbb N}, \mathcal B(E)^{\otimes\mathbb N}, \mathbb P_{\mu})$ with respect to the natural filtration and its square integrable martingale differences are of the form
$$
Z_n=\chi (X_n)-\chi (X_{n-1})+\psi(X_n)\quad\text{for $n\ge 1$}.
$$
Observe that $\mathbb E_{\mu_*} Z_1^2<\infty$. Indeed, we easily check that $x\to \mathbb E_x (Z_1^2\wedge k)$ for any $k\ge 1$ is a bounded continuous function. Further, since $\mathbb E_{\mu P^n}(Z_1^2\wedge k)=\int_E \mathbb E_x(Z_1^2\wedge k)\mu P^n(\d x)\to \mathbb E_{\mu_*} (Z_1^2\wedge k)$ for any $k\ge 1$ as $n\to\infty$ and $\sup_{n\ge 0}\mathbb E_{\mu P^n}(Z_1^2)<\infty$, we obtain that $\mathbb E_{\mu_*} (Z_1^2)<\infty$. 

Set
$$
\sigma^2:=\mathbb E_{\mu_*} Z_1^2.
$$
We have 
\begin{equation}\label{WLLN}
\lim_{n\to\infty}\mathbb E_{\mu P^n} Z_1^2=\lim_{n\to\infty}\mathbb E_{\mu} Z_n^2=\sigma^2.
\end{equation}
In fact, since $\chi$ and $\psi$ are Lipschitzean, we have $\sup_{n\geq 1} \mathbb E_{\mu} |Z_n|^{2+\delta}<\infty$, by (H3). Further, observe that
$$
\sup_{n\ge 1}\mathbb E_{\mu} (Z_n^2 {\bf 1}_{\{|Z_n|^2\ge k\}})\le k^{-\delta/2} \sup_{n\geq 1} \mathbb E_{\mu} |Z_n|^{2+\delta}\to 0
$$
as $k\to \infty$. Therefore, condition (\ref{WLLN}) follows from the fact that $\mathbb E_{\mu P^n}(Z_1^2\wedge k)\to \mathbb E_{\mu_*} (Z_1^2\wedge k)$ as $n\to\infty$ for any $k\ge 1$. Finally, we obtain
$$
\lim_{n\to\infty}\frac{s_n^2}{n}=\lim_{n\to\infty}\frac{\mathbb E_{\mu} S_n^2}{n}=\lim_{n\to\infty}\frac{\sum_{i=1}^n \mathbb E_{\mu} Z_i^2}{n}=\sigma^2.
$$

\begin{lemma}
The square integrable martingale differences $(Z_n)_{n\ge 1}$ satisfy the following condition:
\begin{equation}
\frac{1}{n}\sum_{l=1}^n Z_l^2\to \sigma^2\quad \text{$\mathbb P_{\mu}$-a.s. as $n\to\infty$}
\end{equation}
and consequently if $\sigma^2>0$ condition (\ref{e3}) holds.
\end{lemma} 
\begin{proof} First observe that to finish the proof it is enough to show that for any $i\in\{1,\ldots, n_0\}$ we have
$$
\frac{1}{n}\sum_{l=1}^n Z_{i+l n_0}^2\to \sigma^2\quad \text{$\mathbb P_{\mu}$-a.s. as $n\to\infty$}.
$$

 If we show that both the functions
$$
x\to \mathbb E_x(|\liminf_{n\to\infty}(1/n\sum_{l=1}^n Z_{i+l n_0}^2)-\sigma^2| \wedge 1)
$$
and
$$
x\to \mathbb E_x(|\limsup_{n\to\infty}(1/n\sum_{l=1}^n Z_{i+l n_0}^2)-\sigma^2| \wedge 1)
$$
are continuous, we shall be done. Indeed, then we have
$$
\begin{aligned}
&\mathbb E_{\mu}(|\liminf_{n\to\infty}(1/n\sum_{l=1}^n Z_{i+l n_0}^2)-\sigma^2| \wedge 1)=\int_E \mathbb E_x(|\liminf_{n\to\infty}(1/n\sum_{l=1}^n Z_{i+l n_0}^2)-\sigma^2| \wedge 1)\mu(\d x)\\
&=\int_E \mathbb E_x(|\liminf_{n\to\infty}(1/n\sum_{l=1}^n Z_{i+l n_0}^2)-\sigma^2| \wedge 1)\mu P^{i+m n_0}(\d x)\to \mathbb E_{\mu_*}(|\liminf_{n\to\infty}(1/n\sum_{l=1}^n Z_{i+l n_0}^2)-\sigma^2| \wedge 1),
\end{aligned}
$$
as $m\to+\infty$, by the fact that $\mu P^{i+m n_0}$ converges weakly to $\mu_*$ as $m\to+\infty$. On the other hand, from the Birkhoff individual ergodic theorem we have 
$$
\mathbb E_{\mu_*}(|\liminf_{n\to\infty}(1/n\sum_{l=1}^n Z_{i+l n_0}^2)-\sigma^2| \wedge 1)=0
$$
and consequently
$$
\mathbb E_{\mu}(|\liminf_{n\to\infty}(1/n\sum_{l=1}^n Z_{i+l n_0}^2)-\sigma^2| \wedge 1)=0,
$$
which, in turn, gives
$$
\liminf_{n\to\infty}(1/n\sum_{l=1}^n Z_{i+l n_0}^2)=\sigma^2 \quad \text{$\mathbb P_{\mu}$-a.s.}
$$
Analogously we may show that
$$
\limsup_{n\to\infty}(1/n\sum_{l=1}^n Z_{i+l n_0}^2)=\sigma^2 \quad \text{$\mathbb P_{\mu}$-a.s.}
$$
The remainder of the proof is devoted to showing the continuity of the relevant functions. Again, we restrict to the first function, since the proof for the second one goes in almost the same manner.

Observe that
$$
\begin{aligned}
&\mathbb E_{x}(|\liminf_{n\to\infty}(1/n\sum_{l=1}^n Z_{i+l n_0}^2)-\sigma^2| \wedge 1)\\
&=\lim_{n \rightarrow \infty}  \lim_{k \rightarrow \infty} \mathbb E_{x}\left(\left|  \min \left\{  \frac{1}{n} \sum_{l=1}^{n} Z_{i+l n_0}^2-\sigma^2 , \ldots , \frac{1}{n+k} \sum_{l=1}^{n+k} Z_{i+l n_0}^2 -\sigma^2  \right\}  \right|\wedge 1\right)\\
&=\lim_{n \rightarrow \infty}  \lim_{k \rightarrow \infty} H_{n, k} (x),
\end{aligned}
$$
where
$$
\begin{aligned}
&H_{n, k}(x)=\mathbb E_x(|  \min \{  {1}/{n} ( \sum_{l=1}^{n} Z_{i+ln_0}^2  \wedge n(1+\sigma^2) )-\sigma^2 , \ldots , \\
&{1}/{(n+k)} ( \sum_{l=1}^{n+k} Z_{i+l n_0}^2 \wedge (n+k)(1+\sigma^2)-\sigma^2 ) \} |\wedge 1)\\
&=\mathbb E_x(|  \min \{  {1}/{n} ( \sum_{l=1}^{n} (\chi (X_{i+l n_0})-\chi (X_{i-1+l n_0})+\psi(X_{i+l n_0}))^2  \wedge n(1+\sigma^2) )-\sigma^2 , \ldots , \\
&{1}/{(n+k)} ( \sum_{l=1}^{n+k} (\chi (X_{i+l n_0})-\chi (X_{i-1+l n_0})+\psi(X_{i+l n_0}))^2 \wedge (n+k)(1+\sigma^2)-\sigma^2 ) \} |\wedge 1).\\
\end{aligned}
$$
Set
$$
\begin{aligned}
g_{n, k}(y_1, &\ldots, y_{2(n+k)})
=|  \min \{  {1}/{n} ( \sum_{l=1}^{n} (\chi (y_{2l})-\chi (y_{2l-1})+\psi(y_{2l}))^2  \wedge n(1+\sigma^2) )-\sigma^2 , \ldots , \\
&{1}/{(n+k)} ( \sum_{l=1}^{n+k} (\chi (y_{2l})-\chi (y_{2l-1})+\psi(y_{2l}))^2 \wedge (n+k)(1+\sigma^2)-\sigma^2 ) \} |\wedge 1
\end{aligned}
$$
so that
$$
H_{n,k}(x) = \mathbb E_x(g_{n, k}(X_{i+n_0-1}, X_{i+n_0}, X_{i+2n_0-1}, X_{i+2n_0}, \ldots, X_{i+2(n+k)n_0-1}, X_{i+2(n+k)n_0})).
$$ 
Observe that $H_{n, k}$ is given by formula (\ref{e1l}). If we show that there exists $L$ such that $g_{n, k}$ is Lipschitz continuous in each variable with the Lipschitz constant $L$ (independent of $n, k$), then all $H_{n, k}$ are Lipschitzean with the same Lipschitz constant $\tilde L$, by Lemma 1. Consequently $\lim_{n \rightarrow \infty}  \lim_{k \rightarrow \infty} H_{n, k} $ is Lipschitzean and in particular continuous. Since minimum of any finite family of functions which are Lipschitz continuous in each variable with the Lipschitz constant $L$ is Lipschitz continuous in each variable with the same Lipschitz constant $L$, to finish the proof it is enough to observe that the function
$$
(y_1, \ldots, y_{2p})\to{1}/{p} ( \sum_{l=1}^{p} (\chi (y_{2l})-\chi (y_{2l-1})+\psi(y_{2l}))^2  \wedge p(1+\sigma^2) )-\sigma^2
$$
is Lipschitz continuous in each variable with the Lipschitz constant $L$ for fixed $L>0$. On the other hand, each term in the above sum is Lipschitz continuous in each variable with the Lipschitz constant
$$
(1/p) (\Lip \chi +\Lip\psi) 2p (1+\sigma^2)=2(\Lip \chi +\Lip\psi)(1+\sigma^2).
$$
Observe that each variable appears in one term in the above sum. Hence $L\le  2(\Lip \chi +\Lip\psi)(1+\sigma^2)$, which finishes the proof.
\end{proof}

 Note that following the proof of our previous lemma we are able to show that the considered Markov chain satisfies the strong law of large numbers {\bf (SLLN)}. This result however directly follows from Theorem 2.1 in \cite{Sh}.

\begin{lemma}
Let $\sigma^2>0$. 
Under the assumptions (H0)--(H3) the square integrable martingale differences $(Z_n)_{n\ge 1}$ satisfy conditions (\ref{e1}), (\ref{e2}).
\end{lemma}

\begin{proof} Since $\sup_{n\geq 1} \mathbb E_{\mu} |Z_n|^{2+\delta}<\infty$, $\delta$ is the constant given in (H3), we have
$$
\sum_{n=1}^{\infty} s_n^{-4}\mathbb E_{\mu}[Z_n^4{\bf 1}_{\{|Z_n|<\gamma s_n\}}]\le\sum_{n=1}^{\infty} s_n^{-4}\gamma^{2-\delta} s_n^{2-\delta} \mathbb E_{\mu} |Z_n|^{2+\delta}\le\gamma^{2-\delta}\sup_{n\ge 1}\mathbb E_{\mu} |Z_n|^{2+\delta}\sum_{n=1}^{\infty} s_n^{-2-\delta}.
$$
On the other hand, the condition $s_n^2/n\to\sigma^2$ as $n\to \infty$ gives $\sum_{n=1}^{\infty} s_n^{-2-\delta}<\infty$, which completes the proof of condition (\ref{e1}).

To show condition (\ref{e2}) observe that
$$
\sum_{n=1}^{\infty} s_n^{-1}\mathbb E_{\mu}[|Z_n|{\bf 1}_{\{|Z_n|\ge\epsilon s_n\}}]\le\sum_{n=1}^{\infty} s_n^{-1}\mathbb E_{\mu}[|Z_n|^{2+\delta}/(\epsilon s_n)^{1+\delta}]\le\epsilon^{-1-\delta}\sup_{n\ge 1}\mathbb E_{\mu} |Z_n|^{2+\delta}\sum_{n=1}^{\infty} s_n^{-2-\delta}<\infty.
$$
The proof is complete.
\end{proof}

 \subsection{The Law of Iterated Logarithm for Markov chains}

 \begin{theorem}
 Let $(X_n)_{n\ge 0}$ be a Markov chain  with an initial distribution $\mu$ satisfying conditions (H0)--(H3). If $\psi$ is a Lipschitz function with $\langle \psi, \mu_*\rangle =0$ and $\sigma^2>0$, then $\mathbb P_{\mu}$-a.s. the sequence
 $$
\theta_n(t)= \frac{\sum_{i=1}^k\psi(X_i)+(nt-k) \psi(X_{k+1})}{\sigma\sqrt{2n\log\log n}}
$$ 
if $k\le n t\le k+1$, $k=1, \ldots, n-1$ for $t>0$, $n> e$ and $\theta_n(t)=0$ otherwise is relatively compact in $C$ and the set of its limit points coincides with $K$.
\end{theorem}

 \begin{proof} First observe that since $s_n^2/n\to \sigma^2>0$ as $n\to\infty$ we have 
 $$
 \frac{\sqrt{2s_n^2\log\log s_n^2}}{\sigma\sqrt{2n\log\log n}}\to 1\quad\text{as $n\to\infty$.}
 $$
 Consequently, from Lemmas 2 and 3 it follows that the sequence
 $$
\eta_n(t)=\frac{S_k+(s_n^2t-s_k^2)(s_{k+1}^2-s_k^2)^{-1} Z_{k+1}}{\sigma\sqrt{2n\log\log n}}
$$ 
if $s_k^2\le s_n^2 t\le s_{k+1}^2$, $k=1, \ldots, n-1$ for $t>0, n>e$ and $\eta_n(t)=0$ otherwise is relatively compact in $C$ and the set of its limit points coincides with $K$, due to Heyde and Scott \cite{HS}. Let $t\in (0, 1]$ and $n\ge 1$. Observe that 
if $k\le n t \le k+1$, then
$$
\frac{k\sigma^2}{s_k^2} s_k^2\le \frac{n\sigma^2}{s_n^2} t s_n^2\le \frac{(k+1)\sigma^2}{s_{k+1}^2} s_{k+1}^2.
$$
Set
$$
\hat{\eta}_n(t)=\frac{S_k+(nt-k)Z_{k+1}}{\sigma\sqrt{2n\log\log n}},
$$
where $k\ge 1$ such that $k\le  n t\le k+1$. 
Since $(n\sigma^2)/s_n^2\to 1$ as $n\to\infty$ for any $\varepsilon>0$ holds
$$
(1-\varepsilon)s_k^2\le (1+\varepsilon) s_n^2 t\le (1+\varepsilon)^2 (1-\varepsilon)^{-1}s_{k+1}^2
$$
for all $n$ large enough. Hence there is $t_*\in [t(1-\varepsilon)(1+\varepsilon)^{-1}, t (1+\varepsilon)(1-\varepsilon)^{-1}]$ such that $s_k^2\le s_n^2 t_*\le s_{k+1}^2$. On the other hand, the diameter of the interval 
$[s_k^2/s_n^2, s_{k+1}^2/s_n^2]$ for a fixed $k=1, \ldots, n-1$ converges to $0$ as $n\to \infty$.
Consequently, for any $t>0$ and $n>e$ there exists $t_n>0$ such that
$\hat{\eta}_n(t)=\eta_n(t_n)$ and $t_n\to t$ as $n\to\infty$. Since the sequence $(\eta_n(t))_{n> e}$ is relatively compact in $C$ and the set of its limit points coincides with $K$, the sequence $(\hat{\eta}_n(t))_{n>e}$ is also relatively compact and has the same set of limits points.

Fix $\varepsilon>0$. Define the sets
$$
A_n=\left\{\omega\in\Omega: \frac{|S_n-\sum_{i=1}^n\psi (X_i)|}{\sqrt{n}}\ge\varepsilon/2\right\}\cup\left\{\omega\in\Omega: \frac{|Z_n-\psi (X_n)|}{ \sqrt{n}}\ge\varepsilon/2\right\}\quad\text{for $n\ge 1$}.
$$
Now we are going to show that $\sum_{n=1}^{\infty}\mathbb P_{\mu}(A_n)<\infty$. Indeed, keeping in mind that $\chi$ is Lipschitzean, by the Chebyshev inequality we obtain
$$
\begin{aligned}
&\mathbb P_{\mu}\left(\left\{\omega\in\Omega: \frac{|S_n-\sum_{i=1}^n\psi (X_i)|}{ \sqrt{n}}\ge\varepsilon/2\right\}\right)=\mathbb P_{\mu}\left(\left\{\omega\in\Omega: \frac{|\chi(X_n)-\chi(X_0)|}{ \sqrt{n}}\ge\varepsilon/2\right\}\right)\\
&\le c_0\frac{\mathbb E (\rho_{x_0}(X_n))^{2+\delta}+\mathbb E (\rho_{x_0}(X_0))^{2+\delta}}{n^{1+\delta/2}}\le \frac{c}{n^{1+\delta/2}},
\end{aligned}
$$
by (H3) for some constant $c>0$ independent of $n$.

Analogously, we may check that there exists a positive constant $C$ (independent of $n$) such that
$$
\begin{aligned}
&\mathbb P_{\mu}\left(\left\{\omega\in\Omega: \frac{|Z_n-\psi (X_n)|}{\sqrt{n}}\ge\varepsilon/2\right\}\right)=\mathbb P_{\mu}\left(\left\{\omega\in\Omega: \frac{|\chi(X_n)-\chi(X_{n-1})|}{\sqrt{n}}\ge\varepsilon/2\right\}\right)\\
&\le \frac{C}{n^{1+\delta/2}},
\end{aligned}
$$
by (H3) and the Lipschitz property of the function $\chi$. Thus the series $\sum_{n=1}^{\infty}\mathbb P_{\mu}(A_n)$ is convergent.

Finally, from the Borel--Cantelli lemma it follows that $\mathbb P_{\mu}$-a.s.
$$
\limsup_{n\to\infty}\sup_{0\le t\le 1}\left|\frac{S_k+(nt-k)Z_{k+1}}{\sigma\sqrt{2n\log\log n}}-\frac{\sum_{i=1}^k\psi(X_i)+(nt-k)\psi(X_{k+1})}{\sigma\sqrt{2n\log\log n}}\right|<\varepsilon,
$$
where $k\le nt\le k+1$. Since $\varepsilon>0$ was arbitrary, the proof is complete.
 \end{proof}

\end{document}